\documentclass[11pt,reqno]{amsart}
\usepackage{geometry}
\usepackage{amssymb, amsmath, amsthm, color, tikz,enumerate,mathrsfs}
\usepackage{pdfpages}
\usepackage{hyperref}
\usepackage{textcomp}
\usepackage{color}
\usepackage{xcolor}
\usepackage{listings}

\usepackage{fullpage}

\newtheorem{lemma}{Lemma}[section]
\newtheorem{proposition}[lemma]{Proposition}
\newtheorem{theorem}[lemma]{Theorem}
\newtheorem{corollary}[lemma]{Corollary}

\theoremstyle{definition}

\makeatletter
\newcommand{\setword}[2]{%
  \phantomsection
  #1\def\@currentlabel{\unexpanded{#1}}\label{#2}%
}
\makeatother

\numberwithin{equation}{section}

\newcommand{\Rrr}{\mathbb{R}}

\newcommand{\bl}[2]{(#1,#2)}
\newcommand{\bigbl}[2]{\left(#1,#2\right)}

\DeclareMathOperator{\trace}{tr}
\DeclareMathOperator{\Vol}{Vol}

\makeatletter
\@namedef{subjclassname@2020}{%
  \textup{2020} Mathematics Subject Classification}
\makeatother

\begin{document}

\title{Volumes of consecutively defined sets}

\author{Richard Ehrenborg and Evan Henning}

\address{Department of Mathematics, University of Kentucky, Lexington,
  KY 40506-0027, USA.\hfill\break \tt http://www.math.uky.edu/\~{}jrge/,
  richard.ehrenborg@uky.edu.}

\address{Department of Mathematics, University of Kentucky, Lexington,
  KY 40506-0027, USA.\hfill\break \tt 
  ehe294@uky.edu.}


\date{\today.}

\begin{abstract}
We study a variant of the graph polytopes of a path
and of a cycle
where we replace the inequality
$x_{i} + x_{i+1} \leq 1$
with the two inequalities
$(1-\alpha) \cdot x_{i} + \alpha \cdot x_{i+1} \leq \alpha$
for $0 \leq x_{i} \leq \alpha$
and
$\alpha \cdot x_{i} + (1-\alpha) \cdot x_{i+1} \leq \alpha$
for $\alpha \leq x_{i} \leq 1$.
Using a self-adjoint operator
and its eigenvalues we obtain
convergent series for their volumes.
As a corollary we obtain that the volumes
of the set associated to a path on $n$ vertices
and
the set associated to a cycle on $n+1$ vertices
are related by a constant factor of $\alpha$.
\end{abstract}

\maketitle

\section{Introduction}

The two polytopes
\begin{align}
&
\{(x_{1},x_{2}, \ldots, x_{n}) \in [0,1]^{n} :
x_{i} + x_{i+1} \leq 1 \text{ for } 1 \leq i \leq n-1\} , 
\label{equation_P_1/2}
\\
&
\{(x_{1},x_{2}, \ldots, x_{n}) \in [0,1]^{n} :
x_{i} + x_{i+1} \leq 1 \text{ for } 1 \leq i \leq n\} ,
\label{equation_Q_1/2}
\end{align}
where $x_{n+1} = x_{1}$,
have been well studied in the literature.
See the papers~\cite{Beukers_Calabi_Kolk,
Ehrenborg_cycle,
Ehrenborg_Levin_Readdy,
Elkies,
Simons_Yao}.
These two polytopes can also be described as the graph polytopes
of a path, respectively, a cycle;
see the concluding remarks at the end of the paper.
They are closely related to alternating permutations
and hence the Euler numbers.
In this paper we introduce a generalization of these polytopes by
replacing the inequality between two adjacent variables
with the pair of inequalities
\begin{align*}
\frac{1-\alpha}{\alpha} \cdot x_{i} + x_{i+1} & \leq 1,
& & \text{for } 0 \leq x_{i} \leq \alpha, \\
x_{i} + \frac{1-\alpha}{\alpha} \cdot x_{i+1} & \leq 1,
& & \text{for } \alpha \leq x_{i} \leq 1 .
\end{align*}
See Figure~\ref{figure} for a visualization of these
inequalities.
When $\alpha = 1/2$ they yield the inequality $x_{i} + x_{i+1} \leq 1$
and hence the two polytopes~\eqref{equation_P_1/2}
and~\eqref{equation_Q_1/2}.
Note that when $1/2 < \alpha < 1$ our new sets are also
polytopes, whereas for the interval $0 < \alpha < 1/2$
the sets are non-convex.

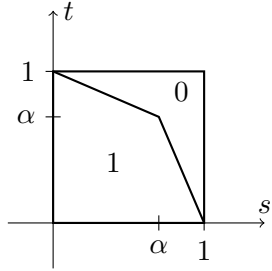
\begin{figure}[b]
\begin{center}
\begin{tikzpicture}[scale = 2]
\draw[->] (-0.3,0) -- (1.4,0) node[above]{$s$};
\draw[->] (0,-0.3) -- (0,1.4) node[right]{$t$};
\draw[thick] (0,0) rectangle (1,1);
\draw[-,thick] (0,1) -- (0.7,0.7) -- (1,0);
\draw[-] (0.05,1) -- (-0.05,1) node[left]{$1$};
\draw[-] (0.05,0.7) -- (-0.05,0.7) node[left]{$\alpha$};
\draw[-] (1,0.05) -- (1,-0.05) node[below]{$1$};
\draw[-] (0.7,0.05) -- (0.7,-0.05) node[below]{$\alpha$};
\node at (0.4,0.4) {$1$};
\node at (0.85,0.87) {$0$};
\end{tikzpicture}
\end{center}
\caption{The symmetric function $\chi$.}
\label{figure}
\end{figure}

To be more explicit,
let $\ell_{1}(t)$ and $\ell_{2}(t)$
be the two affine linear functions
\begin{align*}
\ell_{1}(t)
& =
1-\frac{1-\alpha}{\alpha} \cdot t ,
&
\ell_{2}(t)
& =
\alpha \cdot \frac{1-t}{1-\alpha} .
\end{align*}
Observe that they are inverses of each other.
Also note that
$\ell_{1}(0) = 1$,
$\ell_{1}(\alpha) = \ell_{2}(\alpha) = \alpha$
and $\ell_{2}(1) = 0$.
That is, $\ell_{1}$ maps the interval $[0,\alpha]$
to the interval $[\alpha,1]$.
Similary, $\ell_{2}$ maps the interval~$[\alpha,1]$
back to the interval~$[0,\alpha]$.
Define the self-inverse function $b(t)$ by
\begin{align}
b(t)
& = 
\begin{cases}
\ell_{1}(t) & \text{if } 0 \leq t \leq \alpha, \\
\ell_{2}(t) & \text{if } \alpha \leq t \leq 1.
\end{cases}
\label{equation_b}
\end{align}
Now define the two $n$-dimensional sets $P_{n}$ and $Q_{n}$ by
\begin{align*}
P_{n}
& =
\{(x_{1},x_{2}, \ldots, x_{n}) \in [0,1]^{n} :
x_{i+1} \leq b(x_{i}) \text{ for } 1 \leq i \leq n-1\} , \\
Q_{n}
& =
\{(x_{1},x_{2}, \ldots, x_{n}) \in [0,1]^{n} :
x_{i+1} \leq b(x_{i}) \text{ for } 1 \leq i \leq n\} ,
\end{align*}
where we use that $x_{n+1} = x_{1}$ in the definition of the set $Q_{n}$.
Again note how the set $P_{n}$ corresponds to a path and $Q_{n}$ to a cycle.

In Theorem~\ref{theorem_volumes_P_Q} we express the
volumes of these sets in terms of convergent series.
Our technique is to reformulate the problem into
understanding a linear operator $T$ and its spectrum.
This linear operator is on the space $L^{2}([0,1])$
and is the continuous analogue of a transfer matrix.
That is, the kernel $\chi(s,t)$ plays the role of the adjacency matrix,
$T$ replaces multiplication with this matrix,
the inner product $(1,T^{n-1}(1))$
is the continuous analogue of path enumeration
and
the trace $\trace(T^{n})$
is the cycle enumeration.
In Section~\ref{section_general} we introduce the
operator in a more general case,
where we replace the function $b(t)$ with
a more general self-inverse function $c(t)$.
In Section~\ref{section_alpha} we obtain the
spectrum for our operator in the piecewise linear function case
and hence the sought-after volumes of our two sets.
A particularly striking consequence is
Corollary~\ref{corollary_factor_alpha}
which states that the volumes of $P_{n}$ and~$Q_{n+1}$
differ by a multiplicative constant $\alpha$.
We also use an identity of Lipschitz to give an explicit expression
for the volume of the set $Q_{n}$.
We end the paper with open questions and directions
for future research.

\section{The general case of a self-inverse function}
\label{section_general}

In this section we take a more general approach in understanding the
volumes of the two sets $P_{n}$ and $Q_{n}$.
We introduce a linear operator $T$ on the space $L^{2}([0,1])$,
that is, the space of square integrable functions on the interval $[0,1]$.
The role of the operator $T$ is analogous to a transfer operator for consecutive constraints.
Repeated application of this operator builds admissible sequences of coordinates,
so the volume of $P_{n}$ is obtained from $T^{n-1}$,
while the volume of $Q_{n}$ is the trace of $T^{n}$.
Because the kernel is symmetric, $T$ is self-adjoint and can be diagonalized.
Furthermore, by the spectral theorem, the operator only has
real eigenvalues.

Let $c: [0,1] \longrightarrow [0,1]$ be a decreasing
continuous function that is its own inverse.
That is, we have
$c(0) = 1$,
$c(1) = 0$
and
$c(c(t)) = t$.
Let $\chi(s,t)$ be a function on
the unit square $[0,1]^{2}$ defined by
\begin{align*}
\chi(s,t)
& =
\begin{cases}
1 & \text{ if } s \leq c(t) , \\
0 & \text{ otherwise. }
\end{cases}
\end{align*}
Apply the decreasing function $c$ to the inequality $s \leq c(t)$
and we obtain $c(s) \geq c(c(t)) = t$. 
By a second application of the function $c$ we have that the two inequalities are equivalent.
Hence the function $\chi$ is symmetric,
that is,
$\chi(s,t) = \chi(t,s)$.
Let $P_{n}$ and $Q_{n}$ be the two $n$-dimensional sets
defined by
\begin{align*}
P_{n}
& =
\{(x_{1},x_{2}, \ldots, x_{n}) \in [0,1]^{n} :
x_{i+1} \leq c(x_{i}) \text{ for } 1 \leq i \leq n-1\} , \\
Q_{n}
& =
\{(x_{1},x_{2}, \ldots, x_{n}) \in [0,1]^{n} :
x_{i+1} \leq c(x_{i}) \text{ for } 1 \leq i \leq n\} .
\end{align*}
Note that when the function $c(t)$ is concave down
the two sets $P_{n}$ and $Q_{n}$ are convex.
Observe that the two sets introduced in the introduction is
a special case of this setup.

Define the operator $T$ on the space $L^{2}([0,1])$ by
\begin{align*}
T(f)(t)
& = 
\int_{0}^{1} \chi(t,s) \cdot f(s) \: ds 
=
\int_{0}^{c(t)} f(s) \: ds .
\end{align*}
The operator $T$ is a Hilbert--Schmidt operator
with kernel $\chi$.
That should be viewed as the continuous analogue
of matrix times vector multiplication
where the function $\chi(t,s)$ corresponds to the matrix.
Furthermore, since the function $\chi$ is symmetric,
the operator $T$ is self-adjoint.
Again, the discrete analogue is a symmetric matrix.
The spectral theorem~\cite[XIII.4 Theorem~2]{Dunford_Schwartz}
implies that
all the eigenvalues are real and
the eigenfunctions form a complete orthogonal set.

We also observe that $0$ is not an eigenvalue of the operator~$T$,
since if we have $T(f) = 0$ then differentiating
this equation yields
$c^{\prime}(t) \cdot f(c(t)) = 0$
and hence $f(t) = 0$.

The relation between $T$ and the set $P_{n}$
is demonstrated by the following (small) example,
where we let $1$ denote the constant function on the interval $[0,1]$:
\begin{align*}
\Vol(P_{3})
& =
\int_{0}^{1} \int_{0}^{1} \int_{0}^{1} \chi(x_{1},x_{2}) \cdot \chi(x_{2},x_{3}) \: dx_{3}dx_{2}dx_{1} \\
& =
\int_{0}^{1} \int_{0}^{1} \chi(x_{1},x_{2}) \cdot
\left(\int_{0}^{1} \chi(x_{2},x_{3}) \cdot 1 \: dx_{3} \right)
dx_{2}dx_{1} \\
& =
\int_{0}^{1} \int_{0}^{1} \chi(x_{1},x_{2}) \cdot
T(1)(x_{2}) \:
dx_{2}dx_{1} \\
& =
\int_{0}^{1} T^{2}(1)(x_{1}) \: dx_{1} \\
& =
(1, T^{2}(1)) ,
\end{align*}
where $(\cdot,\cdot)$ denotes the inner product on the space $L^{2}([0,1])$.
Recall that the trace of a matrix is the sum of the entries on the diagonal.
Similarly for a Hilbert--Schmidt operator $U$
with kernel $\kappa$ on the space $L^{2}([0,1])$,
that is,
$U(f)(t) = \int_{0}^{1} \kappa(t,s) \cdot f(s) \: ds$,
the trace is defined by
\begin{align*}
\trace(U)
& =
\int_{0}^{1} \kappa(t,t) \: dt .
\end{align*}
We now express the volume of the set $Q_{3}$ in terms of the trace
of the operator $T^{3}$.
First note that $T^{3}$ is given by
\begin{align*}
T^{3}(f)(x_{1})
& =
\int_{0}^{1} \int_{0}^{1} \int_{0}^{1} \chi(x_{1},x_{2}) \cdot \chi(x_{2},x_{3}) \cdot \chi(x_{3},x_{4})
\cdot f(x_{4})
\: dx_{4}dx_{3}dx_{2} 
\end{align*}
and hence we have
\begin{align*}
\Vol(Q_{3})
& =
\int_{0}^{1} \int_{0}^{1} \int_{0}^{1} \chi(x_{1},x_{2}) \cdot \chi(x_{2},x_{3}) \cdot \chi(x_{3},x_{1})
\: dx_{3}dx_{2}dx_{1}
=
\trace(T^{3}) .
\end{align*}
Recall that the trace of a matrix is also given by the sum of the eigenvalues.
An operator is trace class if the sum of its eigenvalues converges absolutely.
In this case the trace is also given by the sum of the eigenvalues.
\begin{theorem}
Let $n$ be greater than or equal to $2$.
Let $K$ be an index set for the eigenfunctions~$\varphi_{k}$ of the operator $T$
with eigenvalues $\lambda_{k}$.
Then the volume of the set $P_{n}$ is given by
\begin{align}
\Vol(P_{n})
& =
\sum_{k \in K}
\frac{\bl{1}{\varphi_{k}}^{2}}{\bl{\varphi_{k}}{\varphi_{k}}}
\cdot
\lambda_{k}^{n-1} 
\label{equation_volume_P}
\end{align}
Under the assumption that the series
$\sum_{k \in K} \lambda_{k}^{n}$ converges absolutely
the volume of the set $Q_{n}$ is given by
\begin{align}
\Vol(Q_{n})
& =
\sum_{k \in K}
\lambda_{k}^{n} .
\label{equation_volume_Q}
\end{align}
\label{theorem_volumes}
\end{theorem}
\begin{proof}
Begin by observing that
\begin{align*}
T^{m}(f)(t)
& =
\int_{[0,1]^{m}}
\chi(t,s_{m}) \cdot \chi(s_{m},s_{m-1}) \cdots \chi(s_{2},s_{1})
\cdot f(s_{1}) 
\: ds_{1} \cdots ds_{m} .
\end{align*}
Expand the constant function $1$ 
in the complete orthogonal basis
$\{\varphi_{k}\}_{k \in K}$:
\begin{align*}
1
& =
\sum_{k \in K}
\frac{\bl{1}{\varphi_{k}}}{\bl{\varphi_{k}}{\varphi_{k}}}
\cdot
\varphi_{k} .
\end{align*}
Now the volume of the set $P_{n}$ is given by
\begin{align*}
\Vol(P_{n})
& =
\int_{0}^{1}
\int_{[0,1]^{n-1}}
\chi(t,s_{n-1}) \cdot \chi(s_{n-1},s_{n-2}) \cdots \chi(s_{2},s_{1})
\: ds_{1} \cdots ds_{n-1} \: dt \\
& =
\bl{1}{T^{n-1}(1)} \\
& =
\bigbl{1}{
\sum_{k \in K}
\frac{\bl{1}{\varphi_{k}}}{\bl{\varphi_{k}}{\varphi_{k}}}
\cdot
\lambda_{k}^{n-1}
\cdot
\varphi_{k} 
} \\
& =
\sum_{k \in K}
\frac{\bl{1}{\varphi_{k}}^{2}}{\bl{\varphi_{k}}{\varphi_{k}}}
\cdot
\lambda_{k}^{n-1} .
\end{align*}

Recall that $T^{n}$ is a Hilbert--Schmidt operator.
Since the sum $\sum_{k} \lambda_{k}^{n}$
converges absolutely by the assumption of the theorem,
the operator~$T^{n}$ is a trace class operator;
see
Section~XI.8, Exercise~49 in~\cite{Dunford_Schwartz}.
The trace of the operator is given by
the convergent sum
$$  \trace(T^{n}) = \sum_{k \in K} \lambda_{k}^{n}  .  $$
By part (c) of the above mentioned exercise in~\cite{Dunford_Schwartz},
the trace of $T^{n}$ is also given by
the following integral:
\begin{align*}
\trace(T^{n})
& = 
\int_{[0,1]^{n}} 
\chi(s_{1},s_{2}) \cdots \chi(s_{n-1},s_{n}) \cdot \chi(s_{n},s_{1})
\: ds_{1} \cdots ds_{n}
= 
\Vol(Q_{n}) .
\qedhere
\end{align*}
\end{proof}

Note that equation~\eqref{equation_volume_P}
is reminiscent of Theorem~1.1
in~\cite{Ehrenborg_Kitaev_Perry}.
However, since the operator here is self-adjoint
we obtain an exact result.
Equation~\eqref{equation_volume_Q}
is similar to Theorem~3.2 in~\cite{Ehrenborg}.

\begin{proposition}
Assume that $c(t)$ is a piecewise twice differentiable function.
Let $\varphi(t)$ be an eigenfunction
of the operator $T$ with eigenvalue~$\lambda$.
Then the eigenfunction $\varphi(t)$ satisfies
the second order linear differential equation
\begin{align*}
\varphi''(t)
-
\frac{c''(t)}{c'(t)} \cdot \varphi'(t)
-
\frac{c'(t)}{\lambda^{2}} \cdot \varphi(t)
=
0,
\end{align*}
and satisfies the boundary conditions
$\varphi(1) = 0$ and $\varphi'(0) = 0$.
\label{proposition_general_2nd_order}
\end{proposition}
\begin{proof}
The eigenfunction equation states that
\begin{align}
\lambda \cdot \varphi(t)
& =
\int_{0}^{c(t)} \varphi(s) \: ds .
\label{equation_general_defining}
\end{align}
Differentiate this equation twice
\begin{align}
\lambda \cdot \varphi'(t)
& =
c'(t) \cdot \varphi(c(t)) , 
\label{equation_general_varphi'} \\
\lambda \cdot \varphi''(t)
& =
c''(t) \cdot \varphi(c(t))
+
c'(t)^{2} \cdot \varphi'(c(t)) .
\label{equation_general_varphi''}
\end{align}
Substitute $t \longmapsto c(t)$ in
equation~\eqref{equation_general_varphi'}
and we obtain
$\lambda \cdot \varphi'(c(t))
= c'(c(t)) \cdot \varphi(c(c(t)))
= c'(c(t)) \cdot \varphi(t)$.
Applying this identity and~\eqref{equation_general_varphi'}
in equation~\eqref{equation_general_varphi''} yields
\begin{align*}
\lambda \cdot \varphi''(t)
& =
c''(t) \cdot \frac{\lambda}{c'(t)} \cdot \varphi'(t)
+
c'(t)^{2} \cdot \frac{c'(c(t))}{\lambda} \cdot \varphi(t) .
\end{align*}
The derivative of $c(c(t)) = t$ is $c'(t) \cdot c'(c(t)) = 1$.
Applying this last identity yields the second order
differential equation.
Finally, setting $t$ to be $1$
in~\eqref{equation_general_defining}
and $t$ to be $0$ in~\eqref{equation_general_varphi'}
yields the two boundary conditions.
\end{proof}

\section{The operator}
\label{section_alpha}

We now return to our piecewise linear function $b(t)$ from 
equation~\eqref{equation_b} in the introduction.
We determine the eigenvalues explicitly for its associated operator $T$.
The main tool is the second order differential equation in
Proposition~\ref{proposition_general_2nd_order}.
In our case it breaks into two second order differential equations
with constant coefficients.

\begin{proposition}
The eigenvalues of the operator $T$ are given by
\begin{align}
\lambda_{k}
& =
\frac{\sqrt{(1-\alpha) \cdot \alpha}}
{
\arctan\left(\sqrt{\frac{1-\alpha}{\alpha}}\right)
+
\pi \cdot k
} ,
\label{equation_eigenvalue}
\end{align}
where $k$ is an integer.
The eigenfunction associated with the eigenvalue $\lambda_{k}$ is
\begin{align*}
\varphi_{k}(t)
& =
\begin{cases}
\frac{1}{\sqrt{\alpha}}
\cdot
\cos\left( \sqrt{\frac{1-\alpha}{\alpha}} \cdot \frac{1}{\lambda_{k}} \cdot  t \right)
&
\text{if } 0 \leq t \leq \alpha , 
\\
\frac{1}{\sqrt{1-\alpha}}
\cdot
\sin\left( \sqrt{\frac{\alpha}{1-\alpha}} \cdot \frac{1}{\lambda_{k}} \cdot (1-t) \right)
&
\text{if } \alpha \leq t \leq 1 . 
\end{cases}
\end{align*}
Furthermore, the algebraic multiplicity of each of the non-zero
eigenvalues is $1$.
\end{proposition}
\begin{proof}
Let $\varphi$ be an eigenfunction,
that is, $\lambda \cdot \varphi = T(\varphi)$.
Assume that $\varphi$ is of the form
\begin{align*}
\varphi(t)
& =
\begin{cases}
p(t) & \text{if } 0 \leq t \leq \alpha , \\
q(t) & \text{if } \alpha \leq t \leq 1 .
\end{cases}
\end{align*}
The derivative of $b(t)$ is given by
\begin{align*}
b'(t)
& = 
\begin{cases}
-(1-\alpha)/\alpha & \text{ for } 0 \leq t \leq \alpha, \\
-\alpha/(1-\alpha) & \text{ for } \alpha \leq t \leq 1 .
\end{cases}
\end{align*}
Note also that $b''(t) = 0$.
Hence Proposition~\ref{proposition_general_2nd_order}
implies that
\begin{align*}
p''(t) 
+
\frac{1}{\lambda^{2}} \cdot \frac{1-\alpha}{\alpha} \cdot p(t) 
& =
0 ,
&&
0 \leq t \leq \alpha,
\\
q''(t)
+
\frac{1}{\lambda^{2}} \cdot \frac{\alpha}{1-\alpha} \cdot q(t) 
& =
0 ,
&&
\alpha \leq t \leq 1.
\end{align*}
These two second order linear differential equations
have the general solutions
\begin{align}
p(t)
& =
A \cdot \cos\left( \sqrt{\frac{1-\alpha}{\alpha}} \cdot \frac{1}{\lambda} \cdot  t \right)
+
B \cdot \sin\left( \sqrt{\frac{1-\alpha}{\alpha}} \cdot \frac{1}{\lambda} \cdot t \right) ,
\label{equation_p_explicit_A_B}
\\
q(t)
& =
C \cdot \cos\left( \sqrt{\frac{\alpha}{1-\alpha}} \cdot \frac{1}{\lambda} \cdot  t \right)
+
D \cdot \sin\left( \sqrt{\frac{\alpha}{1-\alpha}} \cdot \frac{1}{\lambda} \cdot t \right) .
\label{equation_q_explicit_C_D}
\end{align}
It remains to find the four constants $A$, $B$, $C$ and $D$
and an equation for the eigenvalue $\lambda$.

Note that $q(1) = \varphi(1) = 0$
and $p'(0) = \varphi'(0) = 0$
by
Proposition~\ref{proposition_general_2nd_order}.
This last observation implies $B=0$.
Now see that $A = 0$
would imply that $\varphi(t) = 0$.
Hence we obtain that the constant $A$ is non-zero.
Furthermore, setting $t=1$ in
equation~\eqref{equation_q_explicit_C_D}
yields
\begin{align*}
\tan\left( \sqrt{\frac{\alpha}{1-\alpha}} \cdot \frac{1}{\lambda} \right)
& =
- \frac{C}{D} .
\end{align*}
Thus we may select
\begin{align*}
C
& =
\frac{1}{\sqrt{1-\alpha}}
\cdot
\sin\left( \sqrt{\frac{\alpha}{1-\alpha}} \cdot \frac{1}{\lambda} \right),
&
D
& =
- \frac{1}{\sqrt{1-\alpha}}
\cdot
\cos\left( \sqrt{\frac{\alpha}{1-\alpha}} \cdot \frac{1}{\lambda} \right) .
\end{align*}
We are selecting the coefficients such that the eigenfunction will have
unit length; see Lemma~\ref{lemma_unit_length}.
Now $q(t)$ has the expression
\begin{align}
q(t)
& =
\frac{1}{\sqrt{1-\alpha}}
\cdot
\sin\left( \sqrt{\frac{\alpha}{1-\alpha}} \cdot \frac{1}{\lambda} \cdot (1-t) \right) .
\label{equation_q_explicit_again}
\end{align}

Since $\varphi(t)$ is continuous we have
the condition $p(\alpha) = q(\alpha)$.
From equations~\eqref{equation_p_explicit_A_B}
and~\eqref{equation_q_explicit_again}
we have
\begin{align*}
A \cdot \cos\left( \sqrt{(1-\alpha) \cdot \alpha} \cdot \frac{1}{\lambda} \right)
& =
\frac{1}{\sqrt{1-\alpha}}
\cdot
\sin\left( \sqrt{(1-\alpha) \cdot \alpha} \cdot \frac{1}{\lambda} \right) .
\end{align*}
Hence the constant $A$ is given by the tangent function
\begin{align*}
A
& =
\frac{1}{\sqrt{1-\alpha}}
\cdot
\tan\left( \sqrt{(1-\alpha) \cdot \alpha} \cdot \frac{1}{\lambda} \right) .
\end{align*}
Next we compute the integral
\begin{align*}
\int_{0}^{\alpha} p(t) \: dt
& =
A \cdot 
\int_{0}^{\alpha} 
\cos\left( \sqrt{\frac{1-\alpha}{\alpha}} \cdot \frac{1}{\lambda} \cdot  t \right)
\: dt
\\
& =
A \cdot 
\sqrt{\frac{\alpha}{1-\alpha}} \cdot \lambda \cdot
\sin\left( \sqrt{(1-\alpha) \cdot \alpha} \cdot \frac{1}{\lambda} \right) .
\end{align*}
Since
$\lambda \cdot p(\alpha)
= \lambda \cdot \varphi(\alpha)
= \int_{0}^{\alpha} \varphi(t) \: dt
= \int_{0}^{\alpha} p(t) \: dt$
we have
\begin{align*}
\lambda \cdot
A \cdot \cos\left( \sqrt{(1-\alpha) \cdot \alpha} \cdot \frac{1}{\lambda} \right)
& =
A \cdot 
\sqrt{\frac{\alpha}{1-\alpha}} \cdot \lambda \cdot
\sin\left( \sqrt{(1-\alpha) \cdot \alpha} \cdot \frac{1}{\lambda} \right).
\end{align*}
Canceling the factor $A \cdot \lambda$
we obtain
\begin{align*}
\sqrt{\frac{1-\alpha}{\alpha}}
& =
\tan\left( \sqrt{(1-\alpha) \cdot \alpha} \cdot \frac{1}{\lambda} \right) .
\end{align*}
This yields a more explicit expression
$1/\sqrt{\alpha}$ for the constant $A$.
Furthermore we can now solve for the eigenvalue $\lambda$
\begin{align*}
\arctan\left(\sqrt{\frac{1-\alpha}{\alpha}}\right)
+
\pi \cdot k
& =
\sqrt{(1-\alpha) \cdot \alpha} \cdot \frac{1}{\lambda} .
\end{align*}
Solving for $\lambda$ yields~\eqref{equation_eigenvalue}.
Finally, since we obtain a unique eigenfunction,
the geometric multiplicity is $1$ of the non-zero eigenvalues.
Since there are no generalized eigenfunctions,
the algebraic multiplicity is the same as the geometric multiplicity.
\end{proof}

\begin{lemma}
The following two identities hold:
\begin{align*}
\bl{1}{\varphi_{k}}
& = 
\frac{\lambda_{k}}{\sqrt{\alpha}} , &
\bl{\varphi_{k}}{\varphi_{k}}
& = 
1.
\end{align*}
\label{lemma_unit_length}
\end{lemma}
\begin{proof}
Setting $t=0$ in~\eqref{equation_general_defining}
yields the expression for the inner product $\bl{1}{\varphi_{k}}$.
Next we have
\begin{align*}
\bl{\varphi_{k}}{\varphi_{k}}
& =
\int_{0}^{1} \varphi_{k}(t)^{2} \: dt 
=
\int_{0}^{\alpha} p(t)^{2} \: dt 
+
\int_{\alpha}^{1} q(t)^{2} \: dt 
\\
& =
\frac{1}{\alpha} \cdot
\left[
\frac{t}{2}
+
\frac{\sin\left(
2 \cdot \sqrt{\frac{1-\alpha}{\alpha}} \cdot \frac{1}{\lambda_{k}} \cdot t
\right)}
{4 \cdot \sqrt{\frac{1-\alpha}{\alpha}} \cdot \frac{1}{\lambda_{k}}}
\right]_{0}^{\alpha}
+
\frac{1}{1-\alpha} \cdot
\left[
\frac{t}{2}
+
\frac{\sin\left(
2 \cdot \sqrt{\frac{\alpha}{1-\alpha}} \cdot \frac{1}{\lambda_{k}} \cdot (1-t)
\right)}
{4 \cdot \sqrt{\frac{\alpha}{1-\alpha}} \cdot \frac{1}{\lambda_{k}}}
\right]_{\alpha}^{1} 
\\
& =
1 +
\frac{1}{\alpha} \cdot
\left[
\frac{\sin\left(
2 \cdot \sqrt{\frac{1-\alpha}{\alpha}} \cdot \frac{1}{\lambda_{k}} \cdot t
\right)}
{4 \cdot \sqrt{\frac{1-\alpha}{\alpha}} \cdot \frac{1}{\lambda_{k}}}
\right]_{0}^{\alpha}
+
\frac{1}{1-\alpha} \cdot
\left[
\frac{\sin\left(
2 \cdot \sqrt{\frac{\alpha}{1-\alpha}} \cdot \frac{1}{\lambda_{k}} \cdot (1-t)
\right)}
{4 \cdot \sqrt{\frac{\alpha}{1-\alpha}} \cdot \frac{1}{\lambda_{k}}}
\right]_{\alpha}^{1} 
\\
& =
1 +
\frac{1}{\alpha} \cdot
\frac{\sin\left(
2 \cdot \sqrt{\frac{1-\alpha}{\alpha}} \cdot \frac{1}{\lambda_{k}} \cdot \alpha
\right)}
{4 \cdot \sqrt{\frac{1-\alpha}{\alpha}} \cdot \frac{1}{\lambda_{k}}}
-
\frac{1}{1-\alpha} \cdot
\frac{\sin\left(
2 \cdot \sqrt{\frac{\alpha}{1-\alpha}} \cdot \frac{1}{\lambda_{k}} \cdot (1-\alpha)
\right)}
{4 \cdot \sqrt{\frac{\alpha}{1-\alpha}} \cdot \frac{1}{\lambda_{k}}}
\\
& =
1 +
\frac{\sin\left(
2 \cdot \sqrt{\alpha \cdot (1-\alpha)} \cdot \frac{1}{\lambda_{k}}
\right)}
{4 \cdot \sqrt{\alpha \cdot (1-\alpha)} \cdot \frac{1}{\lambda_{k}}}
-
\frac{\sin\left(
2 \cdot \sqrt{\alpha \cdot (1-\alpha)} \cdot \frac{1}{\lambda_{k}}
\right)}
{4 \cdot \sqrt{\alpha \cdot (1-\alpha)} \cdot \frac{1}{\lambda_{k}}} 
= 1,
\end{align*}
where in the fourth step we use that
$1/\alpha \cdot \left[ t/2 \right]_{0}^{\alpha}
+
1/(1-\alpha) \cdot \left[ t/2 \right]_{\alpha}^{1} = 1$
and in the fifth step setting $t=0$ and $t=1$ in the respective antiderivatives yields $0$.
\end{proof}

\begin{theorem}
The volume of the two $n$-dimensional sets $P_{n}$ and $Q_{n}$,
where $n \geq 2$,
is given by the convergent series
\begin{align*}
\Vol(P_{n})
& = 
\frac{1}{\alpha}
\cdot
\sum_{k = -\infty}^{\infty}
\left(
\frac{\sqrt{(1-\alpha) \cdot \alpha}}
{
\arctan\left(\sqrt{\frac{1-\alpha}{\alpha}}\right)
+
\pi \cdot k
}
\right)^{n+1} ,
\\
\Vol(Q_{n})
& =
\sum_{k = -\infty}^{\infty}
\left(
\frac{\sqrt{(1-\alpha) \cdot \alpha}}
{
\arctan\left(\sqrt{\frac{1-\alpha}{\alpha}}\right)
+
\pi \cdot k
}
\right)^{n} .
\end{align*}
\label{theorem_volumes_P_Q}
\end{theorem}
\begin{proof}
Both results follow from Theorem~\ref{theorem_volumes}.
For the second result
observe that the series converges absolutely.
\end{proof}

\begin{corollary}
The volumes of the $n$-dimensional set $P_{n}$
and the $(n+1)$-dimensional set $Q_{n+1}$
satisfy the identity
$\alpha \cdot \Vol(P_{n}) = \Vol(Q_{n+1})$
for $n \geq 2$.
\label{corollary_factor_alpha}
\end{corollary}

The Lipschitz identity~\cite{Lipschitz}
is the following identity.
\begin{theorem}[Lipschitz]
For a real variable $0 < x < 1$ and
$n$ a positive integer greater than $1$
the following identity holds:
\begin{align}
\sum_{k=-\infty}^{\infty} \frac{1}{(x+k)^{n}}
& =
(-1)^{n-1}
\cdot 
\frac{\pi}{(n-1)!}
\cdot
\frac{d^{n-1}}{dx^{n-1}} \cot(\pi \cdot x) . 
\label{equation_Lipschitz}
\end{align}
\label{theorem_Lipschitz}
\end{theorem}

\begin{corollary}
The volume of the $n$-dimensional set $Q_{n}$
is given by
\begin{align*}
\Vol(Q_{n})
& =
\frac{(-1)^{n-1}}{(n-1)!}
\cdot
((1-\alpha) \cdot  \alpha)^{n/2}
\cdot
\cot^{(n-1)}\left(\arctan\left(\sqrt{\frac{1-\alpha}{\alpha}}\right)\right)
\end{align*}
for $n \geq 2$.
\end{corollary}
\begin{proof}
By substituting $y = \pi \cdot x$ in the Lipschitz identity, equation~\eqref{equation_Lipschitz}, becomes
\begin{align*}
\sum_{k=-\infty}^{\infty} \frac{1}{(y + \pi \cdot k)^{n}}
& =
\frac{(-1)^{n-1}}{(n-1)!}
\cdot
\cot^{(n-1)}(y) , 
\end{align*}
where $\cot^{(n-1)}(y)$ denotes the $(n-1)$st derivative of the cotangent function.
The result follows now by setting $y = \arctan(\sqrt{(1-\alpha)/\alpha})$.
\end{proof}

\section{Concluding remarks}

Is there a direct explanation of Corollary~\ref{corollary_factor_alpha}
without first determining the volumes of the two sets?
Most elegant would be a dissection proof between the set
$Q_{n+1}$ and the Cartesian product $[0,\alpha] \times P_{n}$.

When $1/2 < \alpha < 1$, can the $f$-vectors 
of the polytopes $P_{n}$ and $Q_{n}$ be determined?
The original case $\alpha = 1/2$ was studied
for $P_{n}$ in
the two papers~\cite{Chebikin_Ehrenborg,Ehrenborg_Happ},
whereas it was examined for $Q_{n}$ in the paper~\cite{Ehrenborg_cycle}.
When the parameter $\alpha$ is rational,
can the Ehrhart quasi-polynomials 
of $P_{n}$ and $Q_{n}$ be described?

The graph polytope of a graph $G$ on the vertex set
$\{1,2, \ldots, n\}$ is the polytope
$\{\vec{x} \in \Rrr_{\geq 0}^{n} : \forall ij \in E(G) \:\: x_{i}+x_{j} \leq 1\}$.
Hence for a self-inverse function $c(t)$ on the unit interval, what can be said
about the set
\begin{align*}
\{\vec{x} \in \Rrr_{\geq 0}^{n} : \forall ij \in E(G) \:\: x_{i} \leq c(x_{j})\} ?
\end{align*}

There are other natural self-inverse functions on the interval~$[0,1]$
to consider. Here we present three examples.
\begin{itemize}
\item[(i)]
Let $\alpha_{1}, \alpha_{2}, \ldots, \alpha_{m}$ be
$m$ positive real numbers such that
$\alpha_{1} + \alpha_{2} + \cdots + \alpha_{m} = 1$.
Let $\beta_{k}$ be the partial sum
$\alpha_{1} + \alpha_{2} + \cdots + \alpha_{k}$
such that the interval $[0,1]$ is the union
of the intervals
$[\beta_{0},\beta_{1}] \cup [\beta_{1},\beta_{2}] \cup \cdots \cup
[\beta_{m-1},\beta_{m}]$.
Furthermore let $\rho_{k} = {\alpha_{m+1-k}}/{\alpha_{k}}$.
Note that $\rho_{m+1-k} = \rho_{k}^{-1}$.
Define the linear function $\ell_{k}(t)$
such that $\ell_{k}(t)$ maps the interval
$[\beta_{k-1},\beta_{k}]$
to the interval
$[\beta_{m-k},\beta_{m+1-k}]$
by
\begin{align*}
\ell_{k}(t)
& =
\rho_{k} \cdot (\beta_{k-1} - t) + \beta_{m+1-k} .
\end{align*}
Note that $\ell_{k}(t)$ is a decreasing function.
Define the piecewise linear function $c(t)$ by 
$c(t) = \ell_{k}(t)$ on the interval $[\beta_{k-1},\beta_{k}]$.

\item[(ii)]
The M\"obius tranformation
$c(t) = \frac{1 - t}{1 + a \cdot t}$
for $-1 < a$.

\item[(iii)]
Finally, the function $c(t) = \sqrt[q]{1 - t^{q}}$.
\end{itemize}
All three of these examples are self-inverse functions.
Can the volumes of the associated sets
$P_{n}$ and~$Q_{n}$ be computed in these cases?

Finally, are there other functions $c(t)$ such that
the volumes of the two sets~$P_{n}$ and~$Q_{n+1}$ are
related as in Corollary~\ref{corollary_factor_alpha}?

\section*{Acknowledgements}

The first author thanks Cornell University where parts of this paper were written. 
The authors thank the referees and the editor for their comments
on an earlier version of this paper.
This work was partially supported by a grant from the
Simons Foundation (\#854548 to Richard Ehrenborg).

\newcommand{\journal}[6]{{\sc #1,} #2, {\it #3} {\bf #4} (#5), #6.}
\newcommand{\book}[4]{{\sc #1,} ``#2,'' #3, #4.}
\newcommand{\thesis}[4]{{\sc #1,} ``#2,'' Doctoral dissertation, #3, #4.}
\newcommand{\springer}[4]{{\sc #1,} ``#2,'' Lecture Notes in Math.,
                     Vol.\ #3, Springer-Verlag, Berlin, #4.}
\newcommand{\preprint}[3]{{\sc #1,} #2, preprint #3.}
\newcommand{\preparation}[2]{{\sc #1,} #2, in preparation.}
\newcommand{\appear}[3]{{\sc #1,} #2, to appear in {\it #3}}
\newcommand{\submitted}[4]{{\sc #1,} #2, submitted to {\it #3}, #4.}
\newcommand{\JCTA}{J.\ Combin.\ Theory Ser.\ A}
\newcommand{\AdvancesinMathematics}{Adv.\ Math.}
\newcommand{\AdvancesinAppliedMathematics}{Adv.\ in Appl.\ Math.}
\newcommand{\JournalofAlgebraicCombinatorics}{J.\ Algebraic Combin.}

\end{document}